\documentclass[11pt]{amsart}
\usepackage{}
\usepackage{amssymb}
\theoremstyle{plain}
\newtheorem{Thm}{Theorem}

\begin{document}

\title[Bochner formula  on locally finite
graphs ] {Bochner formula, Bernstein type estimates, and porous
media equation on locally finite graphs }

\author{Li Ma}
\address{L. Ma, Distinguished Professor, Department of mathematics \\
Henan Normal university \\
Xinxiang, 453007 \\
China} \email{lma@tsinghua.edu.cn}

\thanks{The research is partially supported by the National Natural Science
Foundation of China No. 11271111 and SRFDP 20090002110019}

\begin{abstract}
In this paper, we consider three typical problems on a locally
finite connected graph. The first one is to study the Bochner
formula for the Laplacian operator on a locally finite connected
graph. We use the Bochner formula to derive the Bernstein type
estimate of the heat equation. The second is to derive the Reilly
type formula of the Laplacian operator. The last one is to obtain
global positive solution to porous-media equation via the use of
Aronson-Benilan argument. There is not much work in the direction of
the study of nonlinear heat equations on locally finite connected
graphs.

{ \textbf{Mathematics Subject Classification 2000}: 05C50,
53Cxx,35Jxx, 68R10}

{ \textbf{Keywords}: Bochner formula, heat equation, global
solution, porous-media equation}
\end{abstract}

 \maketitle

\section{introduction}

In this paper, we study some typical problems related to heat
equations and porous-media equation on a locally finite connected
graph. We do believe that the study of nonlinear heat equations on
locally finite connected graphs is an important subject as like it
is in the Riemannian geometry (see \cite{Do2}). After some thinking,
we immediately realize that the Sobolev type inequality on graphs
\cite{Ch} plays a key role in such a research. However, Sobolev type
inequality on graphs is not a topic of this paper. We first study
the Bochner formula for the Laplacian operator on a locally finite
connected graph. Our Bochner formula is new and should be very
useful in the study of eigenvalue estimate of the Laplacian
operators on graphs. In fact, by invoking the trick of integration
by part on locally finite connected graphs, one may also obtain the
Reilly formula on locally finite connected graphs. Once we have the
Bochner formula, we use it to study the global behavior of the
bounded solution to the heat equation. It is quite nature to ask if
we can get a Bernstein type estimate for the solutions to the heat
equation on locally finite connected graphs. We can obtain this
result. The last question under our consideration of this paper is
to obtain global positive solution to porous-media equation via the
use of Aronson-Benilan argument. This is a hard question since we
may not have the Sobolev compactness imbedding theorem and it is not
easy to obtain the global solution from the exhaustion domain
method. We can overcome this difficulty by using the Aronson-Benilan
type estimate of the bounded solutions to the Porous-media equation.
Our main results are (\ref{bochner}), theorem \ref{key} and theorem
\ref{media} below.

In the previous posted in arxiv draft version of this paper, we
studied the McKean type eigenvalue estimate. Prof. J.Dodzuik
informed me that he and his coauthor had done it in \cite{DL} many
year ago. Here we would like to express my hearty thanks to him for
sending to us their paper \cite{DL}. We may call the locally finite
graph with McKean type inequality the McKean graph. Anyway, it is a
interesting problem to study related heat equation on McKean graphs
and we hope to study it in near future.

Here is the plan of the paper. In section \ref{sect2} we study the
maximum principle of heat equation and for porous media equation. In
section \ref{sect3}, we obtain the Bochner formula  and the
Bernstein type estimate. In the last section \ref{sect5}, we
consider the locally bounded global solution to the Porous-media
equation.

\section{the Maximum principles}\label{sect2}
We start from recalling some definitions and the maximum principle
for bounded solution to heat equation. Let $G=(V(G),E(G))$ be an
infinite, locally finite, connected graph without loops or multiple
edges where $V=V(G)$ is the set of vertices of G and $E=E(G)$ is the
set of edges. We still write $x\in G$ when $x$ is a vertex of $G$.
We use the notation $x \sim y$ to indicate the edge connects the
vertex $x$ to its neighbor vertex $y$. We equip $V$ with the
symmetric weight $\mu_{xy}\geq 0$ associated to the edge $x \sim y$
such that $\sum_{x\sim y}\mu_{xy}>0$ for each $x\in V$ and we always
assume that our edges are unoriented in the sense that
$\mu_{xy}=\mu_{yx}$. We call such a graph the short name the
weighted graph. Let $d_x=\sum_{x\sim y}\mu_{xy}>0$.

We define the space of all square summable functions on $G$,
$$
l^2(V)=\{f:V\to R; \sum_{x\in V} d_xf(x)^2<\infty\}
$$
with the inner product
$$
(f,g)=\sum_{x\in V} d_xf(x)g(x).
$$

Define on $l^2(V)$ the Laplacian operator for the function $f$,
$$
\Delta f(x)=\frac{1}{d_x}\sum_{x\sim y} \mu_{xy}(f(y)-f(x)).
$$
and the norm of the gradient of the function $f$ by
$$
|\nabla f|^2(x)=\frac{1}{d_x}\sum_{x\sim y} \mu_{xy}(f(y)-f(x))^2.
$$

Note that
$$
(\Delta f(x))^2\leq |\nabla f|^2(x).
$$

Fix $x_0\in G$ and let $r(x)=d(x,x_0)$. Let
$$
d_{\pm}(x)=\sum_{\{y\sim x; r(y)=r(x)\pm 1\}} \mu_{xy}.
$$
be the number of vertices which are 1 step closer or further to
$x_0$ than $x$. Define the mean curvature $H(x)$ of the sphere of
radius $r(x)$ to $x_0$ by
$$
H(x)=\Delta r(x)=\frac{1}{d_x}\sum_{x\sim y} \mu_{xy}(r(y)-r(x)).
$$
It can be verified that
$$
H(x)=\frac{d_+(x)-d_-(x)}{d_x}.
$$

Then we have the following maximum principle (\cite{Do}\cite{We}
\cite{W} \cite{MW}).

\begin{Thm}\label{max} Assume that there exists some $x_0\in G$ and a constant
$C\geq 0$ such that $H(x)\leq C$ on $V$. Let $u_0(x)$ be any bounded
function on $G$. Then any bounded solution $u(t,x)$ to the heat
equation
$$u_t=\Delta u$$
with initial data $u(0)=u_0$ satisfies
$$
\sup_G|u(t,x)|\leq \sup_G|u_0(x)|
$$
for every $t\geq 0$.
\end{Thm}

The proof of the result above is standard. By the result above we
can derive the uniqueness of bounded solution to the heat equation
on $V$. In fact the claim follows by considering differences of
bounded solutions with same initial condition. Actually we can
extend the maximum principle to positive solution to the porous
media type equation
\begin{equation}\label{PM}
u_t=\Delta \log u, \ \ (0,\infty)\times V
\end{equation}
with bounded initial data. Again the proof is standard, for
completeness, we give the proof.

\begin{Thm}\label{max2} Assume that there exists some $x_0\in G$ and a constant
$C\geq 0$ such that $H(x)\leq C$ on $V$. Let $v_0(x)$ be any bounded
positive function on $G$. Then any bounded positive solution
$v(t,x)$ to the heat equation (\ref{PM}) with initial data
$v(0)=v_0$ satisfies
\begin{equation}\label{mali}
\sup_G v(t,x)\leq \sup_Gv_0(x)
\end{equation}
for every $t\geq 0$.
\end{Thm}

\begin{proof} By considering $\frac{v(t,x)}{\sup_Gv_0(x)}$ and rescaling the time variable, we may assume $\sup_Gv_0(x)=1$.
Set $u(t,x)=\log v(t,x)$ and $u_0(x)=\log v_0(x)$. Then $u$
satisfies that
\begin{equation}\label{PM2} e^uu_t=\Delta u.
\end{equation}
Let $M_1= \sup \{u(t, x);(t,x)\in (0,T)\times V\}$. Note that $M_2:=
\sup\{u_0(x), x\in V\}=0$.  Clearly we may assume $M_1\geq 0$;
otherwise we are done. We consider for positive $C>0$ and $R>0$ the
function
$$ w(t,x) = u(t, x)-\frac{M_1}{R}(d(x,x_0) + Ct).$$  If we denote
by $B_R = B_R(x_0)$ the ball with radius R and the center $x_0$, we
 may conclude $w(t,x)\leq 0$ for $(t,x)\in \{0\}\times B_R\cup [0,T)\times \partial B_R$,
 which is the parabolic boundary of $[0,T)\times B_R(x_0)$. Assume that $w(t,x)$ attains its
 positive maximum at the interior point $(t_-,x_-)$ of
 $(0,T)\times B_R$, we may assume that
$$w_t>0, \ \ \Delta w\leq 0 $$ at this point. This fact implies that at
$(t_-,x_-)$, $u>0$,
$$
e^uu_t\geq u_t> \frac{M_1C}{R},
$$
and $$ \Delta u\leq\frac{M_1C}{R}.
$$
However, this is impossible due to (\ref{PM2}). Then we have
$w(t,x)\leq0 $ in $[0,T)\times B_R(x_0)$,
 which is
equivalent to $u(t,x)\leq  \frac{M_1}{ R}(d(x,x_0) + Ct)$. Letting
$R\to\infty$ we obtain $u(t,x)\leq 0$ on $[0,T)\times V$, which
gives us (\ref{mali}).
\end{proof}

 The maximum principle above gives us a comparison lemma for the porous media equation (\ref{PM}).
 We shall use this fact in section five.

\section{Bochner formula and Bernstein estimate for heat
equation}\label{sect3}
 Following the method of Bakry-Emery, we
define
$$
\Gamma(f,g)=\frac{1}{2}\{\Delta (fg)(x)-f(x)\Delta g(x)-g(x)\Delta
f(x)\}
$$
and
$$
\Gamma_2(f,g)=\frac{1}{2}\{\Delta\Gamma (fg)(x)-\Gamma(f,\Delta
g)(x)-\Gamma(g,\Delta f)(x)\}.
$$
Then, by direct computation,
$$
\Delta f^2(x)=2f(x)\Delta f(x)+|\nabla f|^2(x),
$$
$$
\Gamma(f,g)(x)=\frac{1}{2d_x}\sum_{y\sim x}
\mu_{xy}(f(y)-f(x))(g(y)-g(x)),
$$
$$
\Gamma(f,f)(x)=\frac{1}{2}|\nabla f|^2(x),
$$
\begin{equation}\label{gamma2}
\Gamma_2(f,f)(x)=\frac{1}{4}|D^2f|^2(x)-\frac{1}{2}|\nabla
f|^2(x)+\frac{1}{2}(\Delta f)^2(x),
\end{equation}
where
$$
|D^2f|^2(x):=\frac{1}{d_x}\sum_{y\sim
x}\frac{\mu_{xy}}{d_y}\sum_{z\sim y} \mu_{yz}|f(x)-2f(y)+f(z)|^2.
$$

We now compute the Bochner formula for the function $f$.
$$
-\Delta |\nabla f|^2(x)=-|D^2f|^2(x)+\frac{2}{d_x}\sum_{y\sim
x}\frac{\mu_{xy}}{d_y}\sum_{z\sim y}
\mu_{yz}(f(x)-2f(y)+f(z))(f(x)-f(y).
$$
Set $$ I=\frac{2}{d_x}\sum_{y\sim x}\frac{\mu_{xy}}{d_y}\sum_{z\sim
y} \mu_{yz}(f(x)-2f(y)+f(z))(f(x)-f(y).
$$
 Note that
\begin{eqnarray*}
I&=2|\nabla f|^2(x)+\frac{2}{d_x}\sum_{y\sim
x}\mu_{xy}(f(x)-f(y)\Delta f(y)
\\ &
=2|\nabla f|^2(x)+\frac{2}{d_x}\sum_{y\sim
x}\mu_{xy}(f(x)-f(y)(\Delta f(y) \\ &-\Delta f(x))+\Delta
f(x)\frac{2}{d_x}\sum_{y\sim x}\mu_{xy}(f(x)-f(y) \\
& =2|\nabla
f|^2(x)+2|\Delta f|^2(x)-2(\nabla f,\nabla \Delta f)(x).
\end{eqnarray*} Then we have the following Bochner formula:
\begin{equation}\label{bochner}
-\Delta |\nabla f|^2(x)=-|D^2f|^2(x)+2|\nabla f|^2(x)+2|\Delta
f|^2(x)-2(\nabla f,\nabla \Delta f)(x).
\end{equation}

We now use this formula to derive the Bernstein type estimate for
the bounded solution $f(t,x)$ to the heat equation
$$
f_t=\Delta f
$$
with initial data $f_0$. Using (\ref{bochner}) we get that
\begin{equation}\label{bh}
(\partial_t-\Delta)(\frac{1}{2}|\nabla
f|^2(t,x))=-\frac{1}{2}|D^2f|^2(t,x)+|\nabla f|^2(x)+|\Delta
f|^2(t,x).
\end{equation}

Assume on $G$ the curvature condition
\begin{equation}\label{bakry}
\Gamma_2(f,f)\geq \frac{1}{m}(\Delta f)^2(x)++\frac{k}{2}|\nabla
f|^2(x),
\end{equation}
for some constants $m>0$ and $k\in \mathbb{R}$.

Using (\ref{gamma2}) we know that
$$
\frac{1}{2}|D^2f|^2(x)\geq (k+1)|\nabla
f|^2(x)+(\frac{2}{m}-1)(\Delta f)^2(x).
$$
Inserting this back to (\ref{bh}) we get that
$$
(\partial_t-\Delta)(\frac{1}{2}|\nabla f|^2(t,x))\leq -k|\nabla
f|^2(t,x)+(2-\frac{2}{m})|\Delta f|^2(t,x).
$$
Using
$$
|\Delta f|^2(x)\leq |\nabla f|^2(x)
$$
we obtain that
$$
(\partial_t-\Delta)(\frac{1}{2}|\nabla f|^2(t,x))\leq (-k
(2-\frac{2}{m})_+)|\nabla f|^2(t,x).
$$
Recall that
$$
(\partial_t-\Delta) f^2(t,x)=-|\nabla f|^2(t,x).
$$
 Then we can choose $\alpha>0$ such that
 $$
(\partial_t-\Delta)(t|\nabla f|^2(t,x))+\alpha f^2(t,x))\leq 0.
 $$
Using the maximum principle we then have
$$
t|\nabla f|^2(t,x))+\alpha f^2(t,x) \leq \alpha \sup_G f_0^2(x)
$$

Since $f(t,x)$ is uniformly bounded in $t$, we know that there
exists $t_k\to\infty$ such that
$$
f(t_k,x)\to f_\infty(x)
$$
for each $x\in G$, and $\lim_{t_k\to\infty}|\nabla f(t_k,x)|^2\to
0$, which implies that $f_\infty(x)=const.$.
 In conclusion we have the result below.

 \begin{Thm}\label{key} Assume that $G$ is a locally finite
 connected graph with curvature condition (\ref{bakry}).
Assume that there exists some $x_0\in G$ and a constant $C\geq 0$
such that $H(x)\leq C$ on $V$. Let $f_0(x)$ be any bounded function
on $G$. Then any bounded solution $f(t,x)$ to the heat equation
$$f_t=\Delta f$$
with initial data $f(0,x)=f_0(x)$ exists globally and there exists
$t_k\to\infty$ such that
$$f(t_k,x)\to f_\infty(x)
$$
where $f_\infty(x)$ is a constant function.
 \end{Thm}

\section{Global solution to the porous-media equation}\label{sect5}

Given any bounded positive function $u_0:V\to \mathbb{R}_+$. We
consider the global existence of the positive solution $u(t,x)$ to
the porous-media equation
\begin{equation}\label{porous} u_t=\Delta \log u, \ \ \ in \ \
(0,\infty)\times V
\end{equation}
with the initial data $u(0)=u_0$. Just like in the Euclidean domain
case, we may define the equation (\ref{porous}) in the distribution
sense (in time variable). Namely for any compact domain supported
function $\phi$ (which is smooth in the t-variable)defined on space
$(0,\infty)\times V$, we have $$ -\int u\phi _t=\int  \log u
\Delta\phi,
$$
where the integration is taken over the space $(0,\infty)\times V$.

Take any finite subgraph $\Omega\subset V$. We may first consider
(\ref{porous}) in $(0,\infty)\times \Omega$ with initial data and
boundary condition $u_0$. Let $f=\frac{1}{2}\log u$. Then $u=e^{2f}$
and it satisfies the equivalent problem
\begin{equation}\label{porous2} e^f (e^f)_t=\Delta f, \ \ \ in \ \
(0,\infty)\times V.
\end{equation}

Actually we can get the local in time solution $u_\Omega$ to
(\ref{porous}) (respectively $f_\Omega=\frac{1}{2}\log u_\Omega$ to
(\ref{porous2})) in $(0,T)\times \Omega$ (for some $T>0$) by using
the discrete Morse flow method \cite{MW2}.

 For $N>1$ an integer and any $T>0$, let
\[
h=T/N, \quad t_n=nh, \quad n=0,1,2,...,N.
\]
Assume that we have constructed $f_j\in L^2(\Omega)$, $0\leq j\leq
n-1$, and $f_{n-1}$ is a minimizer of the functional
\[
I_{n-1}(f)=\frac{1}{2h}\int_\Omega
|e^{f}-e^{f_{n-2}}|^2\,dx+\frac{1}{2}\int_\Omega |\nabla f|^2\,dx
\]
on the space $H=\{f\in L^2(\Omega)\,|\, f-f_0=0, \ on \
\partial\Omega\}$. Note that $H$ is a closed convex subset of $L^2(\Omega)$. Define
\[
I_n(f)=\frac{1}{2h}\int_\Omega
|e^{f}-e^{f_{n-1}}|^2\,dx+\frac{1}{2}\int_\Omega |\nabla f|^2\,dx
\]
on $H$. It is clear that the infimum is finite and, by applying the
Poincar\'e inequality to $f-f_0$ (see \cite{Ch}), that any
minimizing sequence is bounded in $H$.  By the direct method in the
calculus of variations, one concludes that $I_n$ has a unique
minimizer $f_n$ in $H$ which satisfies
\[
\frac{1}{h}\left(e^f-e^{f_{n-1}}\right)e^f =\Delta f
\]
along with the uniform energy bound
\begin{equation}\label{energy}
\frac{1}{2h}\int_\Omega
|e^{f_n}-e^{f_{n-1}}|^2\,dx+\frac{1}{2}\int_\Omega |\nabla
f_n|^2\,dx\leq \frac{1}{2}\int_\Omega |\nabla f_{n-1}|^2\,dx\leq C.
\end{equation}

We define $f_N(t)\in L^2$ for $t\in [0,T]$ such that, for
$n=1,\dots,N$,
\[
f_N(t)=f_n,  \quad t\in[t_{n-1},t_n].
\]
We further define, for $n=1,\dots, N$,
\[
\partial_t e^{f_N}(t)=\frac{1}{h}(e^{f_n}-e^{f_{n-1}}), \quad
t\in [t_{n-1},t_n].
\]
Then $f_N$ satisfies $$ e^{f_N}\partial_t e^{f_N}(t)=\Delta f_N
$$
in $\Omega\times (0,T)$. Note that the energy bound (\ref{energy})
implies that
$$
\int_0^T \int_\Omega e^{2f_N}+\sup_t\int_\Omega |\nabla
f_N|^2\,dx\leq 5C.
$$
We may use the Poincare inequality to get the uniform $L^2(\Omega)$
bound of $\{f_N\}$.
 Taking a subsequence of $\{f_N\}$ that converges in $L_t^\infty
H$, one obtains a limit $f\in L_t^\infty H$ that satisfies
\[
e^f\partial_t e^f =\Delta f
\]
in distribution sense in the domain $\Omega\times (0,T)$.

To get the globally defined solution, we need the linear upper bound
for $u_\Omega=e^{2f}$ and we follow a well-known argument due to
Aronson and Benilan.

Let $\lambda>1$. Define
$$
w_\lambda(t,x)=\lambda u_\Omega(\lambda^{-1}t,x).
$$
Then $w_\lambda(t,x)$ satisfies (\ref{porous}) in $(0,T)\times
\Omega$ with the initial data and boundary condition $\lambda
u_0(x)$, which is bigger than $u_0(x)$. By using the comparison
principle we know that
$$
w_\lambda(t,x)>u_\Omega(t,x)  \ \ \ in \ \ (0,T)\times \Omega.
$$
Set
$$
v_\lambda(t,x)=w_\lambda(t,x)-u_\Omega(t,x).
$$
Then
$$
\frac{\partial}{\partial\lambda}v_\lambda(t,x) \geq 0, \ \ in \ \
0,T)\times \Omega,
$$
or equivalently $u\leq t^{-1}u$ for $u=u_\Omega$, which by
integration, implies that $u_\Omega(t,x)\leq C(1+t)$ where $C>0$ is
a constant depending only on $u_0$. Hence we can extend the solution
$u_\Omega(t,x)$ globally. Take $\Omega=\Omega_j$ where $V=\bigcup
\Omega_j$, $\Omega_j\subset \Omega_{j+1}$ are exhaustion finite
subgraphs of $V$. Then we get a sequence of solutions $\{u_j\}$
defined on $\Omega_j\times (0,\infty)$. By taking diagonal
subsequence  we can get a sub-convergence sequence on any finite
subset of $V$, still denoted by $\{u_j\}$ and a global (locally
bounded) solution $u(t,x)$ of $(\ref{porous})$ with initial data
$u_0$ such that
$$
u(t,x)=\lim_{j\to\infty} u_j(t,x),
$$
locally in $(0,\infty)\times V$. In summary, we then have

\begin{Thm}\label{media} For any bounded positive function $u_0:V\to \mathbb{R}_+$, there
exists a global solution to (\ref{porous}) with initial data $u_0$.
\end{Thm}

The uniqueness question to (\ref{porous}) is a interesting ( may be
very difficult) problem and it can be considered by using the
maximum principle. We leave it open to interesting readers.

\end{document}